\pdfoutput=1
\RequirePackage{ifpdf}
\ifpdf % We~are running pdfTeX in pdf mode
\documentclass[pdftex]{sigma}
\else
\documentclass{sigma}
\fi

\numberwithin{equation}{section}

\newtheorem{thm}{Theorem}[section]

\newtheorem{lem}[thm]{Lemma}
\newtheorem{prop}[thm]{Proposition}
{ \theoremstyle{definition}
\newtheorem{defin}[thm]{Definition}
\newtheorem{rmk}[thm]{Remark} }

\newcommand{\cM}{\mathcal{M}}

\newcommand{\ttau}{\tilde{\tau}}

\newcommand{\hP}{\hat{P}}
\newcommand{\N}{\mathbb{N}}
\newcommand{\Nz}{\N_0}

\newcommand{\Rset}{\mathbb{R}}

\newcommand{\bt}{\boldsymbol{t}}
\newcommand{\bm}{{\boldsymbol{m}}}

\newcommand{\bQ}{{\boldsymbol{Q}}}
\newcommand{\cR}{{\mathcal{R}}}
\newcommand{\hbm}{{{\bm}_{n-1}}}
\newcommand{\tbm}{{{\bm}_{n-2}}}
\newcommand{\hT}{\hat{T}}

\newcommand{\rL}{\mathrm{L}}
\newcommand{\myell}{n}
\newcommand{\supth}{^{\text{th}}}
\DeclareMathOperator{\Wr}{\operatorname{Wr}}

\begin{document}
%\allowdisplaybreaks

\newcommand{\arXivNumber}{2008.02822}

\renewcommand{\PaperNumber}{016}

\FirstPageHeading

\ShortArticleName{Exceptional Legendre Polynomials and Confluent Darboux Transformations}

\ArticleName{Exceptional Legendre Polynomials\\ and Confluent Darboux Transformations}

\Author{Mar\'ia \'Angeles GARC\'IA-FERRERO~$^{\dag^1}$, David G\'OMEZ-ULLATE~$^{\dag^2\dag^3}$ and Robert MILSON~$^{\dag^4}$}

\AuthorNameForHeading{M.\'A.~Garc\'ia-Ferrero, D.~G\'omez-Ullate and R.~Milson}

\Address{$^{\dag^1}$~Institut f\"ur Angewandte Mathematik, Ruprecht-Karls-Universit\"at Heidelberg,\\
\hphantom{$^{\dag^1}$}~Im Neunheimer Feld 205, 69120 Heidelberg, Germany}
\EmailDD{\href{mailto:garciaferrero@uni-heidelberg.de}{garciaferrero@uni-heidelberg.de}}

\Address{$^{\dag^2}$~Departamento de Ingenier\'ia Inform\'atica, Escuela Superior de Ingenier\'ia,\\
\hphantom{$^{\dag^2}$}~Universidad de C\'adiz, 11519 Puerto Real, Spain}
\EmailDD{\href{david.gomezullate@uca.es}{david.gomezullate@uca.es}}
\Address{$^{\dag^3}$~Departamento de F\'isica Te\'orica, Universidad Complutense de Madrid, 28040 Madrid, Spain}

\Address{$^{\dag^4}$~Department of Mathematics and Statistics, Dalhousie University,\\
\hphantom{$^{\dag^4}$}~Halifax, NS, B3H 3J5, Canada}
\EmailDD{\href{mailto:rmilson@dal.ca}{rmilson@dal.ca}}

\ArticleDates{Received September 22, 2020, in final form February 03, 2021; Published online February 20, 2021}

\Abstract{Exceptional orthogonal polynomials are families of orthogonal polynomials that arise as solutions of Sturm--Liouville eigenvalue problems. They generalize the classical families of Hermite, Laguerre, and Jacobi polynomials by allowing for polynomial sequences that miss a finite number of ``exceptional'' degrees. In this paper we introduce a new construction of multi-parameter exceptional Legendre polynomials by considering the isospectral deformation of the classical Legendre operator. Using confluent Darboux transformations and a~technique from inverse scattering theory, we obtain a~fully explicit description of the operators and polynomials in question. The main novelty of the paper is the novel construction that allows for exceptional polynomial families with an arbitrary number of real parameters.}


\Classification{33C47; 34L10; 34A05}

\section{Introduction and main results}

Exceptional orthogonal polynomials (XOPs) are complete families of
orthogonal polynomials that arise as eigenfunctions of a
Sturm--Liouville eigenvalue problem \cite{GKM09}. XOPs are more general
than classical OPs, because the degree sequence of the polynomial
family can have a~finite number of missing, ``exceptional'' degrees.
As in the classical theory, XOPs fall into three broad classes:
Hermite, Laguerre and Jacobi, depending on whether the domain of
orthogonality is the full line, the half-line or a finite
interval~\cite{GFGM19}. Unlike the classical case, the corresponding exceptional
second-order operator has rational rather than polynomial
coefficients.

Exceptional polynomials appear in mathematical physics as bound states
of exactly solvable rational extensions \cite{GGM14,OS09,Qu09} and
exact solutions to Dirac's equation \cite{SH14}. They appear also in
connection with super-integrable systems \cite{MQ13, PTV12} and
finite-gap potentials \cite{HV10}. From a mathematical point of view,
the main results are concerned with the full classification of
exceptional polynomials \cite{GFGM19, GKM12}, properties of their zeros
\cite{GMM13,Ho15,KM15}, and recurrence relations
\cite{Du15,GKKM16,MT15,Od15}.

At the time of this writing, the most general construction of
exceptional Jacobi polyno\-mials~\cite{Bo20, Du17} involves a~finite number of \emph{discrete} parameters and is given in terms of a~Wronskian-like determinant of classical Jacobi polynomials, indexed by two partitions. The
purpose of this note is to show that the class of exceptional
orthogonal polynomials is much richer than previously thought. We do
this by studying the class of exceptional Legendre polynomials, which
cannot be obtained using the standard approach of multi-step Darboux
transformations indexed by partitions. The main novelty of the new
families is that they contain an arbitrary number of \emph{continuous}
deformation parameters. Another innovation is the use of integral
rather than differential operators in the construction of the
exceptional polynomials.

\begin{defin}\label{def:LegOp}
Let $\tau=\tau(z)$ be a polynomial. We say that the operator
\begin{equation} \label{eq:hTdef}
 \hT(\tau) = \big(1-z^2\big)\left( D_z^2 - 2
 \frac{\tau_z}{\tau} D_z +
 \frac{\tau_{zz}}{\tau} \right)
 -2z D_z
\end{equation}
is an \emph{exceptional Legendre operator} if there exist polynomials $\big\{\hP_i(z)\big\}_{i\in\Nz}$ and
constants $\{\lambda_i\}_{i\in\Nz}$, where $\Nz=\{0, 1,\dots\}$, such that
\[ \hT(\tau) \hP_i = \lambda_i \hP_i \] and such that the degree
sequence $\big\{ \deg \hP_i\big\}_{i\in\Nz}$ is missing finitely many
``exceptional'' degrees~\cite{GKM09}.
\end{defin}

Note that, in making this definition, we are \emph{not} assuming that $\deg \hP_i =i$.

\begin{rmk}\label{rmk:orth}
As a direct consequence of this definition, if
$\tau(z)$ has no zeros on $[-1,1]$ and if the eigenvalues are
distinct, then the resulting eigenpolynomials are orthogonal relative
to the inner product
\begin{equation}\label{eq:orth}
 \int_{-1}^1 \frac{\hP_{i_1}(z)\hP_{i_2}(z)}{\tau(z)^2} {\rm d}z = 0,\qquad i_1\neq i_2.
\end{equation}
\end{rmk}
In this case, the eigenpolynomials $\big\{\hat P_i(z)\big\}_{i\in\Nz}$ may define a complete orthogonal polynomial system, which motivates the following definition.
\begin{defin}\label{def:LegPol}
Let $\tau(z)$ be a polynomial that does not vanish in $[-1,1]$. The set $\big\{\hP_i(z)\big\}_{i\in\Nz}$ is a family of \emph{exceptional Legendre polynomials} if
\begin{enumerate}\itemsep=0pt
\item[(i)] $\big\{\hP_i(z)\big\}_{i\in\Nz}$ are eigenfunctions of a Sturm--Liouville problem in $[-1,1]$.
\item[(ii)] $\big\{\deg\hP_i\big\}_{i\in\Nz}$ contains all but finitely many positive integers.
\item[(iii)] The polynomials $\big\{\hP_i(z)\big\}_{i\in\Nz}$ satisfy the orthogonality relation~\eqref{eq:orth}.
\item[(iv)] The polynomials $\big\{\hP_i(z)\big\}_{i\in\Nz}$ form a complete set in the Hilbert space $\rL^2\big([-1,1],\tau^{-2}{\rm d}z\big)$.
\end{enumerate}
\end{defin}

In other words, exceptional Legendre polynomials are just exceptional
polynomials defined in $[-1,1]$ with orthogonality weight
$W(z)=\tau(z)^{-2}$, where $\tau(z)$ is a polynomial not vanishing in
$[-1,1]$. It should be noted, however, that the standard construction
of exceptional Jacobi polynomials based on a multi-index determinant
labelled by two partitions \cite[equation~(5.1)]{Du17} does not allow
parameters $\alpha=\beta=0$ (see \cite[equation~(2.36)]{Bo20}). Thus, the
construction of exceptional Legendre polynomials requires a different
approach, which we present in this paper.

It is known \cite[Theorem~1.2]{GFGM19} that every exceptional operator
can be related to a classical Bochner operator by a finite number of
Darboux transformations. This is true in particular for an
exceptional operator having the form~\eqref{eq:hTdef}, where the
degree of $\tau(z)$ is equal to the number of exceptional degrees. The
new exceptional polynomial families introduced in this paper do not
invalidate the classification result \cite[Theorem~1.2]{GFGM19}, but
rather they highlight the fact that the full class of Darboux
transformations leading to exceptional polynomials is larger than
previously thought.

As a matter of fact, we will consider in this paper a new class of
exceptional operators that are obtained from the classical Legendre operator
\begin{equation} \label{eq:T0def}
 T := \hT(1)= \big(1-z^2\big) D_z^2 -2z D_z
\end{equation}
by the application of a finite number of confluent Darboux transformations (CDTs)~\cite{GQ15}, also known as the ``double commutator'' method~\cite{GT96}. A CDT applied
within a spectral gap of a~second-order self-adjoint operator allows
to add one eigenvalue to the spectrum.
We will relate~$T$ to~$\hT(\tau)$ by a chain of CDTs, but the
commutation procedure we consider is performed at an existing
eigenvalue. The resulting spectral transformation for every confluent pair ``deletes'' an existing eigenvalue and then ``adds'' it back.\footnote{This is true in a formal sense only, as the intermediate potential is singular.}
The overall effect is that of an isospectral transformation~\cite{KSW89}.

An important feature of confluent Darboux transformations is that
every confluent pair of transformations naturally introduces an extra
deformation parameter. Known instances of exceptional Jacobi
polynomials are indexed by discrete parameters and cannot be
continuously deformed into their classical counterparts. By contrast,
after performing $n$ CDTs on the classical Legendre operator~\eqref{eq:T0def} at distinct energy levels indexed by
$\bm=(m_1,\ldots, m_\myell)\in \Nz^\myell$ we will arrive at an
exceptional Legendre operator
\[T_\bm(\bt_\bm) = \hT(\tau_\bm(z;t_\bm))\]
that depends on $n$ real parameters $\bt_{\bm}=(t_{m_1},\ldots, t_{m_\myell}) \in\Rset^n$.
The polynomial eigenfunctions of $T_\bm(\bt_\bm)$ are exceptional Legendre polynomials $\{P_{\bm,i}(z;\bt_\bm)\}_{i\in \Nz}$, which depend on $n$ real parameters $\bt_{\bm}=(t_{m_1},\ldots, t_{m_\myell})$, and can be continuously deformed to the classical Legendre polynomials by letting $\bt_\bm\to 0$.

Adapting certain methodologies from the theory of inverse scattering
\cite{AM80, CASH17, Su85}, we are able to exhibit a determinantal
representation of $\tau_\bm(z;\bt_\bm)$ that is formally similar to the
construction of KdV multi-solitons. The difference here is that,
instead of dressing the zero potential, we isospectrally deform a~particular instance of the Darboux--Poschl--Teller potential
\cite{GB13} by modifying the normalizations of a finite number of the
corresponding bound states. Another feature of our approach is that,
rather than working with a Schr\"odinger operator, we remain in a
polynomial setting by utilizing the gauge and coordinate of the
Legendre operator. The result is a constructive procedure that can be
easily implemented using a computer algebra system.

\subsection{Notation and definitions}
The base case of the construction is the classical Legendre operator~$T$, shown in~\eqref{eq:T0def}, and the classical Legendre polynomials~\cite{szego}
\begin{equation} \label{eq:Pndef}
 P_i(z) := \frac{2^{-i}}{i!} D_z^i \big(z^2-1\big)^i = 2^{-i}\sum_{k=0}^i
 \binom{i}{k}^2 (z-1)^{i-k}(z+1)^{k},\qquad i\in \Nz.
\end{equation}
These classical orthogonal polynomials do have $\deg P_i =i$, they satisfy
the eigenvalue relation
\[ TP_i = -i(i+1)P_i,\qquad i\in \Nz , \] and they form an
$\rL^2$-complete orthogonal family relative to the inner product
\[ \int_{-1}^1 P_{i_1}(z) P_{i_2}(z) {\rm d}z = \frac{2}{2i_1+1} \delta_{i_1i_2},\qquad
 i_1,i_2\in \Nz.\]

Before we can state the main results of the paper, we would like to fix some notation conventions to be used throughout the paper.
Bold symbols such as $\bm$, $\bt$ or $\bQ$ will represent tuples of integers, real numbers or polynomials (one dimensional objects), while calligraphic symbols like~$\cR$ will denote matrices. To access the components of a vector or tensor we will employ square brackets, i.e., $[\cR]_{k\ell}$ denotes the $(k,\ell)$ entry of~$\cR$. In addition, given an $n\times n$ matrix~$\cR$ and integers $1\leq k<\ell\leq n$, we denote by $\cM_{k,\ell}(\cR)$ the $(\ell-k+1)\times(\ell-k+1)$ square submatrix of~$\cR$ that includes the intersection of rows and columns from~$k$ to~$\ell$.

If $\bm=(m_1,\dots,m_n)$ is an $n$-tuple and $k\in\{1,\dots,n\}$, we
will denote by $\bm^{\langle k \rangle}$ the $(n-1)$-tuple where the
element $[\bm]_k$ is removed, i.e., $\bm^{\langle k \rangle}=(m_1,\dots,m_{k-1},m_{k+1},\dots,m_n)$, and
by~$\bm_k$ the $k$-tuple formed by the first $k$ elements of~$\bm$,
i.e., $\bm_k=(m_1,\dots, m_k)$. In particular, we may write explicitly~$\bm_n$ instead of~$\bm$ whenever the context requires to emphasize the length of the tuple, mostly in the proofs by induction or
recurrence relations.

Associated to an $n$-tuple of integers $\bm=(m_1,\dots,m_n)\in\Nz^n$,
we will define the $n$-tuple of real parameters
$\bt_\bm=(t_{m_1},\dots,t_{m_n})\in \Rset^n$. Semicolons will be used
to separate objects of different nature. Commas will be used for
tuple concatenation, e.g., if $i_1,\dots, i_k\in \Nz$,
$(\bm, i_1,\dots,i_k)$ denotes the $(n+k)$-tuple
$(m_1,\dots,m_n, i_1,\dots,i_k,)$. Similarly, we have
$(\bt_{\bm},t_{i_1},\dots, t_{i_k})=\bt_{(\bm,i_1,\dots,i_k)}$. Often,
we will omit the parentheses when denoting $1$-tuples, e.g., $m_1$
instead of~$(m_1)$. Finally, we will use, depending on the context,
the following notation for derivatives of a function $f$ with respect
to~$z$: $D_zf$, $f'$ and~$f_z$.

With this notation in mind, we proceed to define the main objects of this paper.

\begin{defin}\label{def:objects}
Given an $n$-tuple $\bm\in\Nz^n$ and the associated $\bt_\bm\in\Rset^n$, we define
 $\cR_\bm(z;\bt_\bm)$ as the $n\times n$ matrix with polynomial entries given by
\begin{gather*}
[\cR_\bm(z;t_\bm)]_{k\ell}=\delta_{k\ell}+t_{m_\ell} R_{m_km_\ell}(z), \qquad k,\ell\in\{1,\dots, n\},
\end{gather*}
where
\begin{equation} \label{eq:Rijdef}
 R_{m_km_\ell}(z):= \int_{-1}^z P_{m_k}(u) P_{m_\ell}(u) {\rm d}u,
\end{equation}
and $P_i(z)$ denote the classical Legendre polynomials \eqref{eq:Pndef}.
We denote its determinant by
\begin{gather}\label{eq:taudef}
 \tau_{\bm}(z;\bt_\bm) := \det \cR_\bm(z;\bt_\bm).
\end{gather}
We define the $n$-tuple of polynomials
\begin{align}
\label{eq:Qmdef}
 \bQ_{\bm}^{\rm T}(z;\bt_{\bm}):= \tau_{\bm}(z;\bt_{\bm}) \cR_{\bm}(z;\bt_{\bm})^{-1} \big(P_{m_1}(z),\ldots, P_{m_\myell}(z)\big)^{\rm T},
 \end{align}
 Finally, for $i\in\Nz$, we define the polynomials
\begin{align}\label{eq:Pimdef}
P_{\bm;i}(z;\bt_\bm):=\big[\bQ_{(\bm,i)}(z;\bt_{(\bm,i)})\big]_{n+1}.
\end{align}
\end{defin}

Note that, by construction, $\tau_{\bm}(z;\bt_\bm)$ is symmetric in
$\bm$ and $\bQ_{\bm}(z;\bt_{\bm})$ is equivariant with respect to
permutations of $\bm$. In addition, $P_{\bm;i}(z;\bt_\bm)$ is
symmetric in $\bm$ and does not depend on $t_i$ since
$\tau_{(\bm,i)}(z;\bt_{(\bm,i)})
\big[\cR_{(\bm,i)}(z;\bt_{(\bm,i)})^{-1}\big]_{n+1,j}$ correspond to
the minors of the last column of $\cR_{(\bm,i)}(z;\bt_{(\bm;i)})$,
the only column where $t_i$ appears.

For example, for $m_1,m_2\in\Nz$ we have
\begin{gather*}
 \tau_{m_1}(z;t_{m_1})= 1+t_{m_1} R_{m_1m_1}(z),\\
 \cR_{(m_1,m_2)}(z;\bt_{(m_1,m_2)})=
 \begin{pmatrix}
 1+t_{m_1} R_{m_1m_1}(z)& t_{m_2} R_{m_1m_2}(z) \\ t_{m_1}
 R_{m_2m_1}(z)& 1+t_{m_2} R_{m_2m_2}(z)
 \end{pmatrix},\\
 \tau_{(m_1,m_2)}(z;\bt_{(m_1,m_2)})
= 1+t_{m_1} R_{m_1m_1}(z)+
 t_{m_2} R_{m_2m_2}(z)
 \\
 \hphantom{\tau_{(m_1,m_2)}(z;\bt_{(m_1,m_2)})=}{} + t_{m_1}t_{m_2}\big(R_{m_1m_1}(z) R_{m_2m_2}(z) - R_{m_1m_2}^2(z) \big),\\
 \bQ_{(m_1,m_2)}(z;\bt_{(m_1,m_2)}) =
 \begin{pmatrix}
 \big(1+t_{m_2} R_{m_2m_2}(z)\big)P_{m_1}(z) - t_{m_2}
 R_{{m_1}{m_2}}(z)P_{m_2}(z) \\
 \big(1+t_{m_1} R_{{m_1}{m_1}}(z)\big)P_{m_2}(z) - t_{m_1}
 R_{{m_2}{m_1}}(z)P_{m_1}(z)
 \end{pmatrix},\\
 P_{{m_1};i}(z;t_{m_1}) = \big(1+t_{m_1} R_{{m_1}{m_1}}(z)\big)P_i(z) -
 t_{m_1} R_{i{m_1}}(z) P_{m_1}(z), \qquad i\in\Nz.
\end{gather*}

After defining these objects, we are now ready to state the results.
\subsection{Main results}

The main result of this paper states that the polynomials $\{P_{\bm;i}(z;\bt_\bm)\}_{i\in\Nz}\!$ defined by \mbox{\eqref{eq:Rijdef}--\eqref{eq:Pimdef}} are exceptional Legendre polynomials, provided the real parameters $\bt_\bm$ satisfy certain constraints to ensure that $\tau_\bm(z;t_\bm)$ has constant sign on $z\in[-1,1]$.

\begin{thm} \label{thm:Teigen}
 For $\bm\in \Nz^n$, consider the operator
 \begin{equation*}
 T_\bm(\bt_\bm) := \hT(\tau_\bm(z;\bt_\bm)),
 \end{equation*}
given by \eqref{eq:hTdef} and \eqref{eq:taudef}.
 Then $T_\bm(t_\bm)$
 is an exceptional Legendre operator that satisfies
 \begin{equation} \label{eq:Teigen}
 T_\bm(\bt_\bm) P_{\bm;i}(z;\bt_\bm) = - i(i+1) P_{\bm;i}(z;\bt_\bm),\qquad i\in \Nz,
 \end{equation}
 with $P_{\bm;i}(z;\bt_\bm)$ as in~\eqref{eq:Pimdef}.
\end{thm}

In light of \eqref{eq:Teigen}, we may refer to $P_{\bm;i}(z;\bt_\bm)$,
where $i\in \Nz$ varies and $\bm$ and $\bt_\bm$ are fixed, as
exceptional Legendre polynomials. This requires according to Remark~\ref{rmk:orth} and Definition~\ref{def:LegPol} that $\tau_\bm(z;t_\bm)$ does not
vanish on $[-1,1]$. The following theorem gives necessary and
sufficient conditions for this to be true. In that case, like their
classical counterparts, the polynomials $\{P_{\bm;i}(z;\bt_\bm)\}_{i\in\Nz}$ are
orthogonal and complete.

\begin{thm} \label{thm:Portho}
For $\bm\in\Nz^n$ with $m_1,\ldots, m_\myell$ distinct,
the polynomial $\tau_\bm(z;\bt_\bm)$ in~\eqref{eq:taudef} has no zeros on $[-1,1]$ if and
 only if
\begin{equation}\label{eq:tconstraints}
 t_{m_j}>-m_j-\frac 1 2,\qquad j\in \{ 1,\dots,n\}.
 \end{equation}
 If the above conditions hold, then $\{P_{\bm;i}(z;\bt_\bm)\}_{i\in\Nz}$ are exceptional Legendre polynomials with
 \begin{gather*}
 \int_{-1}^1 \frac{P_{\bm;i_1}(u;\bt_\bm)P_{\bm;i_2}(u;\bt_\bm)}{\tau_\bm(z;\bt_\bm)^2} {\rm d}u
 = \frac{2}{1+2i_1 + 2(\delta_{i_1m_1}+ \cdots + \delta_{i_1m_\myell})
 t_{i_1}}\delta_{i_1i_2},\qquad\! i_1,i_2\in\Nz.
\end{gather*}
and $\rL^2$-completeness in $[-1,1]$ relative to the measure $\tau_\bm(z;\bt_\bm)^{-2} {\rm d}z$.
\end{thm}

\begin{rmk}\label{rmk:degenerate}
Note that we could reformulate the above result without the assumption
that $m_1,\ldots, m_\myell$ are distinct. However, there is no extra benefit
in doing this, as demonstrated by Proposition~\ref{prop:duplicates} below.
Assuming that the indices $m_1,\ldots, m_\myell$ are all distinct does not
entail any loss of generality.
\end{rmk}

The rest of the paper is organized as follows: in Section~\ref{sec:1pair} we study exceptional Legendre operators connected by a single step rational confluent Darboux transformation, which involves in fact two Darboux transformations at the same factorization energy. We will iterate these results in Section~\ref{sec:iteration} to consider any number of confluent Darboux transformations and we will relate this construction with the objects in Definition~\ref{def:objects}, thus yielding the proofs of the main theorems.
Finally, in Section~\ref{sec:examples} we give some explicit examples of the new exceptional Legendre families.

\section{One step confluent Darboux transformations}\label{sec:1pair}

In this section we collect a number of relevant Propositions for the
proofs of the Theorems~\ref{thm:Teigen} and~\ref{thm:Portho}. We introduce the concept of a rational confluent Darboux transformation and we show that this transformation preserves the class of exceptional Legendre operators.

Before introducing rational confluent Darboux transformation, we recall $2$-step ordinary Darboux transformations between two operators $T_1$ and $T_2$ with rational coefficients. If $A_1$, $A_2$,~$B_1$,~$B_2$ are first-order differential operators with rational coefficients, and $\lambda_1$, $\lambda_2$ two constants, consider the following $2$-step rational Darboux transformation:
\begin{gather*}
T_1 = B_1A_1+\lambda_1,\\
\tilde T = A_1B_1+\lambda_1 = A_2B_2+\lambda_2,\\
T_2= B_2A_2+\lambda_2.
\end{gather*}
The first transformation at energy level $\lambda_1$ maps $T_1$ to $\tilde T$ and is state-deleting, while the second transformation at energy level $\lambda_2$ maps $\tilde T$ to $T_2$ and is state-adding (or equivalently, the inverse transformation from~$T_2$ to $\tilde T$ is state-deleting). The confluent version arises when $\lambda_1=\lambda_2$, i.e., we use seed functions at each of the two steps which are (at least formally) eigenfunctions of the corresponding (formal) operator with the same eigenvalue. A full discussion of confluent Darboux transformations from this point of view can be seen, for instance, in~\cite{baye}. We can make this notion more precise in the following definition.

\begin{defin} Let $T_1$, $T_2$ be second-order operators with rational coefficients.
 We will say that~$T_1$ and~$T_2$ are related by a \emph{rational confluent Darboux transformation} if there
 exist first-order operators $A_1$, $A_2$, $B_1$, $B_2$, all with rational
 coefficients, and a constant
 $\lambda$ such that
 \begin{gather*} A_1 B_1 = A_2 B_2, \qquad
 T_1= B_1 A_1 + \lambda,\qquad T_2 = B_2 A_2 + \lambda. \end{gather*}
\end{defin}

Given polynomials $\tau(z)$, $\phi(z)$, we define the rational operators
\begin{gather}%\label{eq:Adef}
 A(\tau,\phi) := \tau^{-1} \left( \phi D_z - \phi_z\right),\nonumber\\
 B(\phi,\tau) := A(\phi,\tau)\circ \big(1-z^2\big)
 =\phi^{-1} \left( \big(1-z^2\big)\left( \tau D_z - \tau_z\right)-2z\tau\right). \label{eq:Bdef}
\end{gather}
The form of these operators coincides with the general form of the
first-order operators appearing in factorization of operators given in
\cite[Proposition~3.5]{GFGM19}, with a particular choice that ensures
that operator $\hat T(\tau)=B(\phi,\tau)A(\tau,\phi)$ is in the
natural gauge \cite[Definition 5.1]{GFGM19}. In the proofs, we will
use the fact that, for a given function $f$, we have
\[A(\tau,\phi)f=\tau^{-1}\Wr(\phi, f),\]
where $\Wr$ denotes the Wronskian determinant.

Throughout this section, we consider an exceptional Legendre operator
 $\hat T(\tau)$ with polynomials
$\{\pi_i(z)\}_{i\in \Nz}$ that satisfy
 \begin{equation} \label{eq:Tphilami}
 \hT(\tau)\pi_i = \lambda_i \pi_i,\qquad \lambda_{i_1}\neq
 \lambda_{i_2} \quad \text{if} \quad i_1\neq i_2,\quad i,i_1,i_2\in \Nz.
 \end{equation}
 Our goal is to apply a rational CDT on this operator. To this end,
 for $m\in\Nz$ and $t\in\Rset$, let us define the following objects:
\begin{gather}\label{eq:rhoij}
 \rho_{i_1i_2}(z)
 := \int_{-1}^z\frac{\pi_{i_1}(u)\pi_{i_2}(u)}{\tau(u)^2} {\rm d}u,\qquad
 i_1,i_2\in \Nz,\\
 \label{eq:taujdef}
 \tau_m(z;t) := \tau(z)\left( 1+t\rho_{mm}(z) \right),\\
\label{eq:piijdef}
 \pi_{m;i}(z;t)
 := (1+t\rho_{mm}(z))\pi_i(z)- t \rho_{im}(z)
 \pi_m(z),\qquad i\in \Nz.
\end{gather}
For a lighter notation, we may omit the $t$ dependence and write $\tau_m$ instead of $\tau_m (z;t)$. Note that $\tau$ might already
depend on a number of real parameters, so~$\tau_m$ will depend on the same parameters as~$\tau$, plus an extra parameter~$t$.

\begin{rmk}\label{rem:asum}In the rest of this Section, i.e., for the following four Propositions, we shall assume that $\rho_{i_1i_2}(z)$ defined by \eqref{eq:rhoij} is a rational function that vanishes at $z=-1$ and $\tau_m(z)$ and $\pi_{m;i}(z,t)$ defined by \eqref{eq:taujdef}--\eqref{eq:piijdef} are polynomials in $z$.

If we start from an exceptional Legendre operator \eqref{eq:hTdef} for a given $\tau$ polynomial with eigenpolynomials $\pi_i$, these assumptions are far from obvious by looking at \eqref{eq:rhoij}--\eqref{eq:piijdef}. In the next section we will see that the assumptions hold whenever \eqref{eq:Tphilami} does, i.e., that the rational CDT between exceptional Legendre families is well defined.
\end{rmk}

\begin{prop} \label{prop:RCDT}
 For $m\in \Nz$, let $\rho_{mm}(z)$ and $\tau_m(z)$ be defined by \eqref{eq:rhoij}--\eqref{eq:taujdef} and satisfy the assumptions of Remark~{\rm \ref{rem:asum}}. Then, $\hT(\tau)$ and $\hT(\tau_m)$ are related by a rational confluent Darboux transformation with
 \begin{gather*}
 A(\tau,\pi_m) B(\pi_m,\tau) = A(\tau_m,\pi_m) B(\pi_m,\tau_m),\\
 \hT(\tau) = B(\pi_m,\tau) A(\tau,\pi_m) + \lambda_m,\\
 \hT(\tau_m) = B(\pi_m,\tau_m) A(\tau_m,\pi_m) + \lambda_m.
 \end{gather*}
\end{prop}

\begin{proof}The results follow from direct calculation with the previous definitions.
\end{proof}

The following lemma examines the behaviour at the endpoint $z=-1$ of a combination of these objects, and it will be necessary to prove some of the following propositions.

\begin{lem}\label{lem:1}
Let $\{\pi_i\}_{i\in\Nz}$ be polynomials that satisfy the eigenvalue equation \eqref{eq:Tphilami} with \eqref{eq:hTdef} and let $\rho_{i_1,i_2}(z)$ be the rational functions defined by~\eqref{eq:rhoij} that vanish at $z=-1$. Then,
\[\left. \big(1-z^2\big) \frac{\Wr(\pi_i,\pi_m)}{\tau^2}\right|_{z=-1} =0. \]
\end{lem}
\begin{proof}
 Since $\rho_{im}(z)$ is rational and vanishes at $z=-1$, we can write for given $\alpha,\beta\in\Nz$
 \begin{gather*} \rho_{im}(z)=(1+z)^{1+\alpha}q(z), \qquad q(-1)=C_q\neq 0,\qquad \alpha>0,\\
 \tau(z)=(1+z)^{\beta}p(z),\qquad p(-1)=C_p\neq 0,
 \end{gather*}
 where $q$ is a rational function and $p$ is a polynomial.
 Thus
 \[ \rho_{im}'(-1) = \left.\frac{\pi_i\pi_m}{\tau^2}\right|_{z=-1}=\begin{cases} C_q & \text{if } \alpha=1,\\
 0 & \text{if } \alpha\geq 2. \end{cases} \]
 Using the eigenvalue equation we have
 \[ \big(1-z^2\big)(\pi_i''\pi_m-\pi_i\pi_m'')-2\left( \big(1-z^2\big)\frac{\tau'}{\tau}+z\right)(\pi_i'\pi_m-\pi_i\pi_m')=(\lambda_i-\lambda_m)\pi_i\pi_m. \]
 Evaluating both sides at $z=-1$, we have
 \[ \left.\frac{\Wr(\pi_i,\pi_m)}{\tau^2}\right|_{z=-1} = \left.\frac{(\lambda_i-\lambda_m)}{2(1-2\beta)} \frac{\pi_i\pi_m}{\tau^2}\right|_{z=-1}<\infty. \]
 Since the previous expression is bounded, the desired result is proved.
\end{proof}

The next proposition shows how to build the eigenpolynomials of the transformed operator, by the use of the second order intertwining relations for CDTs
\[ (B_2A_1)T_1 = T_2(B_2A_1),\qquad T_1(B_1A_2)=(B_1A_2)T_2. \]

\begin{prop}\label{prop:piij} For $m\in \Nz$, let $\rho_{im}(z)$, $\tau_m(z)$ and $\pi_{m;i}(z,t)$ be defined by \eqref{eq:rhoij}--\eqref{eq:piijdef} and satisfy the assumptions of Remark~{\rm \ref{rem:asum}}.

Then, for $i\in \Nz$, we have
 \begin{gather}
 (\lambda_{m}-\lambda_{i})\rho_{im}
 = \big(1-z^2\big)\tau^{-1} A(\tau,\pi_{m})\pi_{i},\label{eq:rhoim}\\
 (\lambda_{m}-\lambda_{i}) \pi_{m;i} =B(\pi_m,\tau_m) A(\tau,\pi_m)\pi_i,\label{eq:piim}\\
 \hT(\tau_m) \pi_{m;i} = \lambda_i \pi_{m;i}.\label{eq:hatTtaum}
 \end{gather}
\end{prop}

\begin{proof}We start by noticing that
\begin{align*}
(\lambda_m-\lambda_i)\rho_{im}'=(\lambda_m-\lambda_i)\frac{\pi_i\pi_m}{\tau^2}=\big(\big(1-z^2\big)\tau^{-1} A(\tau,\pi_{m})\pi_{i}\big)',
\end{align*}
where for the last equality we use the eigenvalue equation \eqref{eq:Tphilami}, or equivalently the Sturm--Liouville equation
\begin{gather*}
\left(\big(1-z^2\big)\tau^{-2}\pi_i'\right)'+\big(1-z^2\big)\tau_{zz}\tau^{-3}\pi_i=\lambda_i\tau^{-2}\pi_i.
\end{gather*}

The first result follows by integration since Lemma~\ref{lem:1} ensures that
\begin{gather*}
\big(1-z^2\big)\tau^{-1} A(\tau,\pi_{m})\pi_{i}\big|_{z=-1}=\big(1-z^2\big)\frac{\Wr(\pi_m,\pi_i)}{\tau^2}\bigg|_{z=-1}=0.
\end{gather*}

The second identity follows by direct calculation using the definitions and previous identities. Indeed, using \eqref{eq:Bdef} and \eqref{eq:rhoim} we have
\begin{align*}
B(\pi_m,\tau_m) A(\tau,\pi_m)\pi_i
&=A(\pi_m,\tau_m)\circ\big(1-z^2\big)\left(\frac{(\lambda_{m}-\lambda_{i})\rho_{im}}{\big(1-z^2\big)\tau^{-1}}\right)
\\
&=(\lambda_{m}-\lambda_{i})\pi_m^{-1}\Wr(\tau_m,\tau\rho_{im})\\
&=(\lambda_{m}-\lambda_{i})\pi_m^{-1}\tau^2\Wr(1+t\rho_{mm},\rho_{im})
\end{align*}
and deriving $\rho_{im}$ in \eqref{eq:rhoij} and using \eqref{eq:piijdef} leads to the desired result~\eqref{eq:piim}.

The third identity \eqref{eq:hatTtaum} follows trivially from \eqref{eq:piim} and the intertwining relation
\begin{gather*}
 \hT(\tau_m) B(\pi_m,\tau_m) A(\tau,\pi_m)=B(\pi_m,\tau_m) A(\tau,\pi_m)\hT(\tau).\tag*{\qed}
\end{gather*}\renewcommand{\qed}{}
\end{proof}

The next result derives the transformation rule for $\rho_{i_1i_2}$ under a CDT, and it is key to obtain the norming constants of the transformed polynomials in Proposition~\ref{prop:recortho}.

\begin{prop} \label{prop:rho2step}
 For $m,i_1,i_2\in \Nz$, let $\rho_{i_1i_2}(z)$, $\tau_m(z)$ and $\pi_{m;i}(z,t)$ be defined by \eqref{eq:rhoij}--\eqref{eq:piijdef} and satisfy the assumptions of Remark~{\rm \ref{rem:asum}}.
 Then,
 \begin{gather} \label{eq:rho2step}
 \int_{-1}^z \frac{\pi_{m;i_1}(u;t) \pi_{m;i_2}(u;t)}{\tau_m(u)^2}{\rm d}u =
 \rho_{i_1i_2}(z) - \frac{t \rho_{i_1m}(z) \rho_{i_2m}(z)}{1+t \rho_{mm}(z)},\qquad i_1,i_2\in \Nz.
 \end{gather}
\end{prop}

\begin{proof}\looseness=1 The identity between the derivatives of both sides can be easily proved by direct computa\-tion using the definitions. The desired result follows then by integration since \linebreak \mbox{$\rho_{i_1i_2}(-1)=0$}.
\end{proof}

The last proposition of this section shows that the CDT of an exceptional Legendre family falls into the same class under suitable bound on the introduced parameter~$t$.
\begin{prop} \label{prop:recortho}
Assume that $\tau$ does not vanish in $[-1,1]$ and that $\{\pi_i\}_{i\in\Nz}$ are exceptional Legendre polynomials with
 \[
 \int_{-1}^1 \frac{\pi_{i_1}(u) \pi_{i_2}(u)}{\tau(u)^2} {\rm d}u = \nu_{i_1} \delta_{i_1i_2} ,
\]
 for constants $\nu_i>0$, $i\in \Nz$ and completeness in
 $\rL^2\big([-1,1], \tau^{-2}{\rm d}z\big)$. Let
 $m\in \Nz$ and set
\[
 \nu_{m;i} :=
 \begin{cases}
 \nu_i & \text{if } i\neq m,\\
 (t+ \nu_m^{-1})^{-1} & \text{if } i=m.
 \end{cases}
\]
 Then, $\tau_m(z)>0$ on $[-1,1]$ if and only if $\nu_{m;m}>0$. In that case, the set $\{\pi_{m;i}(z;t)\}_{i\in\Nz}$ is a~family of exceptional Legendre polynomials with
 \begin{equation} \label{eq:pinuij}
 \int_{-1}^1 \frac{\pi_{m;i_1}(u) \pi_{m;i_2}(u)}{\tau_m(u)^2} {\rm d}u =
 \nu_{m;i_1} \delta_{i_1i_2} ,\qquad i_1,i_2,m\in \Nz,
 \end{equation}
 and completeness in
 $\rL^2\big([-1,1], \tau_m^{-2}{\rm d}z\big)$.
\end{prop}

\begin{proof} First, note that \eqref{eq:pinuij} is
 true in a formal sense. By~\eqref{eq:rho2step}, the rational
 function
 \[ \rho_{m;i_1i_2}(z;t):= \rho_{i_1i_2}(z) - \frac{t \rho_{i_1m}(z)
 \rho_{i_2m}(z)}{1+t \rho_{mm}(z)} \] is defined by the integral
 on the l.h.s.\ of~\eqref{eq:rho2step}. Furthermore, since we are
 assuming that
 \[ \rho_{i_1i_2}(1) = \delta_{i_1i_2} \nu_{i_1}, \]
 we have
 \[ \rho_{m;i_1i_2}(1;t) = \delta_{i_1i_2} \left( \nu_{i_1} -
 \delta_{i_1 m} \frac{ t \nu_m^2}{1+t\nu_m} \right) = \delta_{i_1i_2}
 \nu_{m;i_1}.\]

 By \eqref{eq:taujdef}, $\tau_m(z)$ is positive on $z\in[-1,1]$ if
 and only if the same is true for $1+t \rho_{mm}(z)$. Since
 $\rho_{mm}(z)$ is an increasing function, the latter is true if and
 only if $1+t\nu_m>0$. Observe that
\[
 \nu_{m;m}^{-1} = t + \nu_m^{-1} = \nu_m^{-1}(1+t\nu_m).
\]
 Hence $\tau_m(z)$ is positive on $z\in [-1,1]$ if and only if $\nu_{m;m}>0$.

 Finally, we prove completeness. We assume that the eigenpolynomials
 $\{\pi_i(z)\}_{i\in \Nz}$ are $\rL^2$-complete in $[-1,1]$ relative to
 $\tau(z)^{-2}{\rm d}z$. Following an argument adapted from the
 appendix of~\cite{AM80}, we re-express the completeness assumption as
\[
 \sum_{i\in \Nz}\nu_i^{-1}
 \frac{\pi_i(z)}{\tau(z)}\frac{\pi_i(w)}{\tau(w)} = \delta(z-w),
\]
 where the equality is understood in distributional sense on
 $[-1,1]\times [-1,1]$. Rewriting~\eqref{eq:rhoij} as
\[
 \rho_{i_1i_2}(z) = \int_{-1}^1 \theta(z-u)
 \frac{\pi_{i_1}(u)\pi_{i_2}(u)}{\tau(u)^2} {\rm d}u,\qquad
 i_1,i_2\in \Nz ,
\]
where $\theta(z)$ denotes the Heaviside step function, it
follows that for $j\in\Nz$
 \begin{gather*}
 \sum_{i\in \Nz} \nu_i^{-1} \frac{\pi_i(z)}{\tau(z)} \rho_{ij}(w)
 = \theta(w-z) \frac{\pi_j(z)}{\tau(z)},\\
 \sum_{i\in \Nz} \nu_i^{-1} \rho_{ij}(z) \rho_{ij}(w)
 = \theta(w-z) \rho_{jj}(z) + \theta(z-w) \rho_{jj}(w).
 \end{gather*}
 By \eqref{eq:taujdef} and \eqref{eq:piijdef}, we have
 \begin{gather*}
 \frac{\pi_{m;i}(z)}{\tau_m(z)}
 = \frac{\pi_i(z)}{\tau(z)} -t
 \rho_{im}(z) \frac{\pi_m(z)}{\tau_m(z)},\qquad i\in \Nz ,\\
 \frac{t\rho_{mm}(z)}{\tau_m(z)}
 = \frac{1}{\tau(z)}- \frac{1}{\tau_m(z)}.
 \end{gather*}
 Therefore, making use of the previous identities,
 \begin{gather*}
 \sum_{i\in \Nz}\nu_{m;i}^{-1}
 \frac{\pi_{m;i}(z)}{\tau_m(z)}\frac{\pi_{m;i}(w)}{\tau_m(w)}
 == t
 \frac{\pi_{m;m}(z)}{\tau_m(z)}\frac{\pi_{m;m}(w)}{\tau_m(w)}
 + \sum_{i\in \Nz}\nu_{i}^{-1}
 \frac{\pi_{m;i}(z)}{\tau_m(z)}\frac{\pi_{m;i}(w)}{\tau_m(w)}\\
\qquad = t
 \frac{\pi_{m}(z)}{\tau_m(z)}\frac{\pi_{m}(w)}{\tau_m(w)}
 + \sum_{i\in \Nz}\nu_{i}^{-1}
 \left( \frac{\pi_i(z)}{\tau(z)} -t
 \rho_{im}(z) \frac{\pi_m(z)}{\tau_m(z)}\right)
 \left( \frac{\pi_i(w)}{\tau(w)} -t
 \rho_{im}(w) \frac{\pi_m(w)}{\tau_m(w)}\right) \\
 \qquad = t
 \frac{\pi_{m}(z)}{\tau_m(z)}\frac{\pi_{m}(w)}{\tau_m(w)}
 + \delta(z-w) - t\,\theta(w-z) \frac{\pi_m(z)}{\tau(z)}
 \frac{\pi_m(w)}{\tau_m(w)}
 - t\,\theta(z-w) \frac{\pi_m(z)}{\tau_m(z)}
 \frac{\pi_m(w)}{\tau(w)} \\
 \qquad \quad{}+ t^2 \big(\theta(w-z) \rho_{mm}(z) + \theta(z-w)
 \rho_{mm}(w)\big) \frac{\pi_m(z)}{\tau_m(z)}
 \frac{\pi_m(w)}{\tau_m(w)} = \delta(z-w).\tag*{\qed}
 \end{gather*}\renewcommand{\qed}{}
\end{proof}

\section{Recursive construction and proof of theorems}\label{sec:iteration}

The strategy to prove the main theorems is the following. First, we will define some polynomials and rational functions recursively, starting the recursion at the objects corresponding to the classical Legendre Sturm--Liouville problem. The recursion formulas coincide with~\eqref{eq:taudef},~\eqref{eq:piijdef} and~\eqref{eq:rho2step}. Next, we show in Proposition~\ref{prop:taurec} that these recursively defined objects coincide with those defined in Definition~\ref{def:objects}, and thus they satisfy the rationality and polynomiality conditions of Remark~\ref{rem:asum}. Propositions~\ref{prop:RCDT}--\ref{prop:recortho} then ensure that at each step of the recursion we have an exceptional Legendre Sturm--Liouville problem, provided the parameters are chosen in the right range.

\begin{defin}\label{def:recur} \allowdisplaybreaks Let $i,i_1,i_2\in\Nz$,
 $\bm=\bm_n=(m_1,\dots,m_n)\in\Nz^n$ and $\bm_{j}=(m_1,\dots, m_j)$,
 and define recursively functions $R_{\bm;i_1i_2}(z;\bt_{\bm})$,
 $\ttau_{\bm}(z;\bt_{\bm})$ and $\tilde P_{\bm;i}(z;\bt_{\bm})$. For
 $j=0$, we start the recursion at $\bm_0=\varnothing$ with
\begin{gather*}
 R_{\bm_0;i_1i_2}(z;\bt_{\bm_0})=R_{i_1i_2}(z),\\
 \ttau_{\bm_0}(z;\bt_{\bm_0})= 1,\\
 \tilde P_{\bm_0;i}(z;\bt_{\bm_0})=P_i(z),
\end{gather*}
where $R_{i_1i_2}(z)$ are given by \eqref{eq:Rijdef} and $P_i(z)$ are the classical Legendre operators. For $j=1,\dots,n$ we define recursively
 \begin{gather}
 R_{\bm_j;i_1i_2}(z;\bt_{\bm_{j}})=
 R_{\bm_{j-1};i_1i_2}(z;\bt_{\bm_{j-1}})\nonumber\\
 \hphantom{R_{\bm_j;i_1i_2}(z;\bt_{\bm_{j}})=}{} -
 \frac{t_{m_{j}} R_{\bm_{j-1};i_1m_{j}}(z;\bt_{\bm_{j-1}}) R_{\bm_{j-1};i_2m_j}(z;\bt_{\bm_{j-1}})}{
 1+ t_{m_j} R_{\bm_{j-1};m_j m_j}(z;\bt_{\bm_{j-1}})},\label{eq:Rndef}\\
 \label{eq:taurec}
 \ttau_{\bm_j}(z;\bt_{\bm_j})
 = \left( 1+t_{m_j} R_{\bm_{j-1};m_jm_j}(z;\bt_{\bm_{j-1}})\right) \ttau_{\bm_{j-1}}(z;\bt_{\bm_{j-1}}),
 \\
 \tilde{P}_{\bm_j;i}(z;\bt_{\bm_j}) = (1+t_{m_j} R_{\bm_{j-1};m_jm_j}(z;\bt_{\bm_{j-1}}))\tilde{P}_{\bm_{j-1};i}(z;\bt_{\bm_{j-1}}) \nonumber\\
\hphantom{ \tilde{P}_{\bm_j;i}(z;\bt_{\bm_j}) =}{}
 - t_{m_j}R_{\bm_{j-1};im_j}(z;\bt_{\bm_{j-1}}) \tilde{P}_{\bm_{j-1};m_j}(z;\bt_{\bm_{j-1}}). \label{eq:Prec}
 \end{gather}
\end{defin}

The next proposition states that these recursively defined functions coincide with the polynomials and rational functions introduced in Definition~\ref{def:objects}.

\begin{prop}\label{prop:taurec}
 For $i,i_1,i_2\in\Nz$ and $\bm=(m_1,\dots,m_n)\in\Nz^n$, let
 $R_{\bm;i_1i_2}(z;\bt_{\bm})$, $\ttau_{\bm}(z;\bt_{\bm})$ and
 $\tilde P_{\bm;i}(z;\bt_{\bm})$ be the functions defined in
 Definition~{\rm \ref{def:recur}}. Let $\tau_{\bm}(z;\bt_{\bm})$ and
 $ P_{\bm;i}(z;\bt_{\bm})$ be the polynomials defined in
 Definition~{\rm \ref{def:objects}}. Then,
\begin{gather} \label{eq:ttautau}\ttau_{\bm}(z;\bt_{\bm})=\tau_{\bm}(z;\bt_{\bm}),\\
\label{eq:tpp}\tilde P_{\bm;i}(z;\bt_{\bm})=P_{\bm;i}(z;\bt_{\bm}),
\end{gather}
and
\begin{equation} \label{eq:Rint}
 R_{\bm;i_1i_2}(z;\bt_\bm) = \int_{-1}^z
 \frac{P_{\bm;i_1}(u;\bt_\bm)P_{\bm;i_2}(u;\bt_\bm)}{\tau_\bm(u;\bt_\bm)^2}
 {\rm d}u,
 \end{equation}
 where, again, the integral denotes an anti-derivative that vanishes at $z=-1$.
\end{prop}

The consequence of this proposition is to ensure that the rational CDTs applied iteratively on the classical Legendre operator are always well defined, i.e., the conditions specified in Remark~\ref{rem:asum} will hold
at each step of the chain. Before we can address the proof of Proposition~\ref{prop:taurec}, we need to establish the following technical lemma.

\begin{lem}
 \label{prop:Rnrec}
 Let $\bm\in\Nz^n$ with $n\geq 3$. Then
 \begin{gather}
\begin{pmatrix}
 1+t_{m_{n-1}} R_{\tbm;m_{n-1}m_{n-1}}(z;\bt_\tbm) & t_{m_{n}} R_{\tbm;m_{n-1}m_n}(z;\bt_\tbm)\\
 t_{m_{n-1}} R_{\tbm;m_{n-1}m_n}(z;\bt_\tbm) & 1+ t_{m_n} R_{\tbm;m_nm_n;}(z;\bt_\tbm)
 \end{pmatrix}^{-1}\nonumber \\
 \qquad{} = \cM_{n-1,n}\left( \cR_{\bm}(z;t_\bm)^{-1} \right) , \label{eq:M12R}
 \end{gather}
 where $\cM_{n-1,n}$ denotes the bottom right $2\times 2$ submatrix of the
 indicated matrix and $\cR_{\bm}(z;t_\bm)$ is given in Definition~{\rm \ref{def:objects}}.
\end{lem}

\begin{proof} The result follows by iteration and the Sylvester determinant identity \cite{AAM96}. Let us start by showing the argument for $n=3$. If we write
\begin{gather*}
 \cR_\bm(z;\bt_\bm)=
 \left(\begin{array}{@{}c|ccc@{}}
 a &&\boldsymbol{b}_1&
 \\ \hline
 &&&\\
 \boldsymbol{b}_2^{\rm T} && \mathcal A&\\
 &&&
 \end{array}\right),
\end{gather*}
where $\boldsymbol{b}_j$ are 2-tuples for $j=1,2$ and $\mathcal A$ is the $\cM_{2,3}$ submatrix,
by the Sylvester determinant identity,
\begin{gather*}
 \mathcal M_{2,3}\left(\cR_\bm(z;\bt_\bm)^{-1}\right)= \left(\mathcal A-\boldsymbol{b}_2^{\rm T}a^{-1} \boldsymbol{b}_1\right)^{-1}.
\end{gather*}
Identifying the elements in the decomposition of $\cR_\bm(z;\bt_\bm)$ according to its definition, by direct calculation and the recursive formulation \eqref{eq:Rndef} we obtain{\samepage
\begin{gather*}
 \mathcal A-\boldsymbol{b}_1a^{-1} \boldsymbol{b}_2^{\rm T}
 =\begin{pmatrix}
 1+t_{m_{2}} R_{{m_1};m_{2}m_{2}}(z;t_{m_1}) & t_{m_{3}} R_{m_1;m_{2}m_3}(z;t_{m_1})\\
 t_{m_{2}} R_{m_1;m_{2}m_3}(z;t_{m_1}) & 1+ t_{m_3} R_{m_1;m_3m_3}(z;t_{m_1})
 \end{pmatrix}
\end{gather*}
as desired.}

For $n>3$, let $\cR^{[j]}_{\bm}(z;\bt_\bm)$ be the $(n-j)\times(n-j)$ matrix for any $j\in\{1,n-2\}$ with entries
\begin{gather}\label{eq:cRj}
 \big[\cR^{[j]}_{\bm}(z;\bt_\bm)\big]_{k\ell}:=\delta_{k\ell}+t_{m_{j+\ell}}R_{\bm_j;m_{j+k}m_{j+\ell}}(z;\bt_{\bm_j}),\qquad k,\ell\in\{1,\dots,n-j\}.
\end{gather}
We aim to show that
$\cR^{[n-2]}_{\bm}(z;\bt_\bm)^{-1}=\cM_{n-1,n}\big(
\cR_{\bm}(z;t_\bm)^{-1} \big)$. Arguing as above, we have
$\cR^{[n-2]}_{\bm}(z;\bt_\bm)^{-1}=\cM_{2,3}\big(
\cR^{[n-3]}_{\bm}(z;t_\bm)^{-1} \big)$. Applying analogously the
Sylvester determinant identity, we also obtain
\[\cR^{[j+1]}_{\bm}(z;\bt_\bm)^{-1}=\cM_{2,n-j}\big( \cR^{[j]}_{\bm}(z;t_\bm)^{-1} \big)\]
for $j\in\{0,\dots, n-3\}$. The result then follows by iteration.
\end{proof}

\begin{proof}[Proof of Proposition~\ref{prop:taurec}]
Throughout this proof, we are going to omit the explicit dependence on $z$ and $\bt_\bm$ of the objects, which must be understood from the dependence on $\bm$, i.e., we will write $\cR_\bm$ instead of $\cR_\bm(z;\bt_\bm)$.
By definition, we have
\begin{gather*}
 \cR_\bm=\left(
 \begin{array}{@{}ccc|c@{}}
 &&&\\
 &\cR_\hbm && \vdots \\
 &&&\\ \hline
 &\cdots& & 1+t_{m_n} R_{m_nm_n}
 \end{array}\right).
\end{gather*}
Applying the expression for the inverse with the adjoint of the cofactors matrix, we obtain that
\begin{gather*}
 \big[\cR_\bm^{-1}\big]_{nn}=\frac{\det{\cR_\hbm}}{\det{\cR_{\bm_n}}}=\frac{\tau_{\hbm}}{\tau_{\bm_n}}.
\end{gather*}
Applying Lemma~\ref{prop:Rnrec} after computing the inverse of the matrix in the left hand side of~\eqref{eq:M12R} and using the recursion \eqref{eq:Rndef} allows to prove that the $\tau_\bm$ and $\tau_{\bm_{n-1}}$ satisfy the recursion relation~\eqref{eq:taurec}. Since $\tilde\tau_{\bm_1}=\tau_{\bm_1}$, we see that~\eqref{eq:ttautau} holds.

In order to prove \eqref{eq:tpp}, we first observe that \eqref{eq:Prec} can be rewritten as
\begin{gather*}
 \tilde P_{\bm;i}=\frac{\tau_{(\bm;i)}}{\tau_{\hbm}}
 \left[\left(\cR^{[n-1]}_{(\bm,i)}\right)^{-1}
 \big(\tilde{P}_{\hbm;m_n},\tilde{P}_{\hbm;i} \big)^{\rm T}\right]_1,
\end{gather*}
where $\cR^{[n-1]}_{(\bm,i)}=\cR^{[n-1]}_{(\bm,i)}(z;\bt_{(\bm,i)})$ is given by~\eqref{eq:cRj}.
The proof follows by induction. It is clear that~\eqref{eq:tpp} holds for~$\bm_1$ by definition. We assume that it also holds for $\bm_j$ with $j=1, \dots, n-1$, and we must prove it also hods for~$\bm_n$. We start by proving that,
\begin{align}\label{eq:induction}
\left(\cR^{[n-1]}_{(\bm,i)}\right)^{-1}
 \begin{pmatrix}P_{\hbm;m_n}\\P_{\hbm;i} \end{pmatrix}
=\frac{\tau_{\hbm}}{\tau_{\tbm}}
 \left(\left(\cR^{[n-2]}_{(\bm,i)}\right)^{-1}
 \begin{pmatrix}P_{\tbm;m_{n-1}}\\P_{\tbm;m_{n}}\\P_{\tbm;i} \end{pmatrix}\right)^{\langle 1\rangle},
\end{align}
where in the right hand side we have the last two components of a vector of three entries.
This may be verified as follows. First note that, by assumption, we have $\tilde{P}_{\bm_j;i}=P_{\bm_j;i}$ for all $i\in\Nz$ and $j=1,\dots,n-1$. By Lemma~\ref{prop:Rnrec} and \eqref{eq:Prec}, the left hand side is equal to
\begin{gather*}
\cM_{2,3}\left(\left(\cR^{[n-2]}_{(\bm,i)}\right)^{-1}\right)
\begin{pmatrix}
-t_{m_{n-1}}R_{\tbm;m_{n-1},m_n} & \dfrac{\tau_{\hbm}}{\tau_{\tbm}} &0
\\
-t_{m_{n-1}}R_{\tbm;m_{n-1},i}) & 0&
\dfrac{\tau_{\hbm}}{\tau_{\tbm}}
\end{pmatrix}
 \begin{pmatrix}P_{\tbm;m_{n-1}}\\P_{\tbm;m_{n}}\\P_{\tbm;i} \end{pmatrix},
\end{gather*}
where we have used by \eqref{eq:taurec}
\begin{gather*}
 1+t_{m_{n-1}}R_{\tbm;m_{n-1},m_{n-1}}=\frac{\tau_{\hbm}}{\tau_{\tbm}}.
\end{gather*}
Now we can conclude that
\begin{gather}\label{eq:prodmat}
 \cM_{2,3}\left(\left(\cR^{[n-2]}_{(\bm,i)}\right)^{-1}\right)
\begin{pmatrix}
-t_{m_{n-1}}R_{\tbm;m_{n-1},m_n} & \dfrac{\tau_{\hbm}}{\tau_{\tbm}} &0
\\
-t_{m_{n-1}}R_{\tbm;m_{n-1},i} & 0&
\dfrac{\tau_{\hbm}}{\tau_{\tbm}}
\end{pmatrix}
\end{gather}
corresponds to the last two rows of $\left(\cR^{[n-2]}_{(\bm,i)}\right)^{-1}$ multiplied by $\frac{\tau_{\hbm}}{\tau_{\tbm}}$.
For the second block, corresponding to $\cM_{2,3}\left(\left(\cR^{[n-2]}_{(\bm,i)}\right)^{-1}\right)$, the correspondence is clear.
In order to verify the result for the elements $(2,1)$ and $(3,1)$ of $\left(\cR^{[n-2]}_{(\bm,i)}\right)^{-1}$, we identify the components of the second matrix in~\eqref{eq:prodmat} as elements of $\cR^{[n-2]}_{(\bm,i)}$ and we rely on the fact that we are multiplying elements of the matrix $\cR^{[n-2]}_{(\bm,i)}$ and its inverse. In fact, for instance the element $(2,1)$ corresponds to
\begin{gather*}
-\left[\left(\cR^{n-2}_{(\bm,i)}\right)^{-1}\right]_{2,2}
 \left[\cR^{[n-2]}_{(\bm,i)}\right]_{2,1}
 - \left[\left(\cR^{[n-2]}_{(\bm,i)}\right)^{-1}\right]_{2,3}
 \left[\cR^{[n-2]}_{(\bm,i)}\right]_{3,1}\\
 \qquad=-\big[\mathcal{I}\big]_{2,1}+\left[\left(\cR^{[n-2]}_{(\bm,i)}\right)^{-1}\right]_{2,1}
 \left[\cR^{[n-2]}_{(\bm,i)}\right]_{1,1}\\
 \qquad=\big(1+t_{m_{n-1}}R_{\tbm;m_{n-1}m_{n-1}}\big) \left[\left(\cR^{[n-2]}_{(\bm,i)}\right)^{-1}\right]_{2,1}
 =\frac{\tau_{\bm_{n-1}}}{\tau_{\bm_{n-2}}}\left[\left(\cR^{[n-2]}_{(\bm,i)}\right)^{-1}\right]_{2,1},
\end{gather*}
where $\mathcal I$ is the identity matrix.
Similarly to \eqref{eq:induction}, we can show the following identity for ${j=1, \dots, n-2}$:
\begin{gather*}
 \left(\cR^{[j]}_{(\bm,i)}\right)^{-1}
 \begin{pmatrix}P_{\bm_j;m_{j+1}}\\ \vdots\\P_{\bm_j;m_{n}}\\P_{\bm_j;i} \end{pmatrix}
 =\frac{\tau_{\bm_j}}{\tau_{\bm_{j-1}}}\left(\left(\cR^{[j-1]}_{(\bm,i)}\right)^{-1}
 \begin{pmatrix}P_{\bm_{j-1};m_{j}}\\ \vdots\\P_{\bm_{j-1};m_{n}}\\P_{\bm_{j-1};i} \end{pmatrix}\right)^{\langle 1\rangle}.
\end{gather*}

Combining the previous identities yields \eqref{eq:tpp}:
\begin{align*}
 \tilde P_{\bm;i} &=\frac{\tau_{(\bm;i)}}{\tau_{\hbm}}
 \left[\left(\cR^{[n-1]}_{(\bm,i)}\right)^{-1}
 \big(\tilde{P}_{\hbm;m_n},\tilde{P}_{\hbm;i} \big)^{\rm T}\right]_1,
\\& = \frac{\tau_{(\bm;i)}}{\tau_{\bm_{n-2}}}
 \left[\left(\cR^{[n-2]}_{(\bm,i)}\right)^{-1}
 \big({P}_{\bm_{n-2};m_{n-1}},{P}_{\bm_{n-2};m_{n}},{P}_{\bm_{n-2};i} \big)^{\rm T}\right]_1
 =\cdots \\
 & = \tau_{(\bm;i)}
 \left[\left(\cR_{(\bm,i)}\right)^{-1}
 \big({P}_{m_1},\dots {P}_{m_{n}},{P}_{i} \big)^{\rm T}\right]_1
 = P_{\bm;i}.
\end{align*}

Finally, relation \eqref{eq:Rint} follows by Proposition~\ref{prop:rho2step} and by induction on $j$.
\end{proof}

Proposition~\ref{prop:taurec} together with the results in Section~\ref{sec:1pair} allow now to prove the main theorems.

\begin{proof}[Proof of Theorems~\ref{thm:Teigen} and~\ref{thm:Portho}]
 The key to the proof is to observe that starting form the classical Legendre operator, the application of a rational confluent Darboux transformation indexed by an integer $m_i$ introduces an extra real parameter $t_{m_i}$ and leads to a well defined Sturm--Liouville problem defining a family of exceptional Legendre polynomials. Indeed, the equivalence of the objects defined by the CDT recursion \eqref{eq:rhoij}--\eqref{eq:piijdef} and those defined by matrix multiplication in Definition~\ref{def:objects} show that for any $\bm\in\Nz^n$, $\tau_\bm(z;\bt_\bm)$ and $P_{\bm;i}(z;\bt_\bm)$ are polynomials and $\cR_\bm(z;\bt_\bm)$ is a matrix of rational functions that satisfy the premises of Propositions~\ref{prop:RCDT}--\ref{prop:recortho} (see Remark~\ref{rem:asum}). Theorem~\ref{thm:Teigen} follows then from Proposition~\ref{prop:piij} and Theorem~\ref{thm:Portho} follows by induction from Proposition~\ref{prop:recortho}, which establishes the bounds on the parameters~\eqref{eq:tconstraints} that ensure the regularity of $T_\bm(\bt_\bm)$ and the positivity of the measure in $[-1,1]$
\end{proof}

As mentioned above, the degree of the $i$-th exceptional Legendre polynomial $P_{\bm;i}$ indexed by $\bm=(m_1,\dots,m_n)\in\Nz^n$ is not~$i$. The next proposition provides this result.

\begin{prop} \label{prop:degPtau}
 Let $\bm=(m_1,{\ldots}, m_\myell)$ and suppose that $m_1,{\ldots}, m_\myell$ are distinct. Let~$\tau_\bm$,~$P_{\bm;i}$ be as
 defined in \eqref{eq:taudef} and~\eqref{eq:Pimdef}. Then,
 \begin{gather} \label{eq:degtau}
 \deg \tau_{\bm}(z;\bt_\bm) = 2(m_1+\ldots + m_\myell)+\myell,\\
\deg P_{\bm;i}(z;\bt_\bm) = 2(m_1+\ldots + m_\myell)+\myell+i\nonumber\\
 \hphantom{\deg P_{\bm;i}(z;\bt_\bm)=}{} -(\delta_{i,m_1}+ \cdots +
 \delta_{i,m_n})(2i+1),\qquad i\in \Nz. \label{eq:degPn}
 \end{gather}
 Moreover,
 \begin{equation} \label{eq:Pmi}
 P_{\bm;m_k}(z;\bt_\bm) = P_{\bm^{\langle k \rangle};m_k}(z;\bt_{{\bm^{\langle k \rangle}}}) ,\qquad
 k=1,\ldots, \myell,
 \end{equation}
 where $\bm^{\langle k \rangle}\in \Nz^{n-1}$ denotes the tuple obtained by removing the $k\supth$ entry of~$\bm$.
\end{prop}

Notice that \eqref{eq:Pmi} accounts for the above Kronecker delta terms in \eqref{eq:degPn}. It is also worth noting that, as opposed to the ``traditional'' exceptional families, the degree sequence for exceptional Legendre polynomials is not an increasing sequence, which is further evidence of the different construction.
\begin{proof} The identity \eqref{eq:Pmi} follows directly from applying~\eqref{eq:Prec} to this specific choice and the symmetry with respect to permutations in~$\bm$.

Relations \eqref{eq:degtau} and \eqref{eq:degPn} can be proved by induction. It is clear that they hold for $\bm_1=m_1$, since $\deg R_{i_1i_2}(z)=i_1+i_2+1$. Notice that in $P_{m_1;i}(z;t_{m_1})$, the coefficients of $R_{m_1m_1}(z)P_i(z)$ and $R_{m_1i}(z)P_{m_1}(z)$ do not coincide if $i\neq m_1$, so $\deg P_{m_1;i}(z;t_{m_1})=2m_1+i+1$. If $i=m_1$, by~\eqref{eq:Pmi} we obtain $\deg P_{m_1;m_1}(z;t_{m_1})=\deg P_{m_1}(z)=m_1$

Now we assume \eqref{eq:degtau} and \eqref{eq:degPn} hold for $\bm_{j-1}$. By \eqref{eq:Rint}, the degree of $R_{\bm_{j-1};im_{j}}(z;\bt_{\bm_{j-1}})$, understood as the difference between the degree of the polynomial numerator and of the polynomial denominator, is $m_{j}+i+1$ if $i\neq m_1,\dots,m_{j-1}$. Then,
\[ \deg\tau_{\bm_{j}}(z;\bt_{\bm_{j}})=\deg\tau_{\bm_{j-1}}(z;\bt_{\bm_{j-1}})+2m_{j}+1=2(m_1+\dots+m_j)+j.\]
 Arguing as above to verify that there is no cancellation between the highest order contributions, we obtain
\[\deg P_{\bm_{j};i}(z;\bt_{\bm_{j}})=2(m_1+\dots+m_{j-1})+(j-1)+2m_j+i+1\]
 if $i\neq m_1,\dots,m_{j}.$
 In order to prove the result for $i$ equal to some component of $\bm_j$, we employ relation~\eqref{eq:Pmi}.
\end{proof}

\begin{rmk}From Proposition~\ref{prop:degPtau} we see that the codimension (number of missing degrees) of the exceptional Legendre family indexed by $\bm=(m_1,\dots,n)$ is $2(m_1+\cdots+m_n)+n$. This coincides with the degree of~$\tau_\bm$, as it happens for all exceptional polynomials~\cite{GFGM19}.
\end{rmk}

So far we have considered the case when $\bm=(m_1,\dots,m_n)$ contains distinct indices. As announced in Remark~\eqref{rmk:degenerate}, let us show that this choice entails no loss of generality.

\begin{prop} \label{prop:duplicates} Let $\bm\in\Nz^n$ and let $\tau_\bm(z;\bt_\bm)$ and $P_{\bm;i}(z;\bt_\bm)$ be as defined in~\eqref{eq:taudef} and~\eqref{eq:Pimdef}. Then, for any $j\in\Nz$ we have
\begin{gather*}
 \tau_{(\bm,j,j)}\big(z;(\bt_\bm,t_j,t_j)\big) = \tau_{(\bm,j)}\big(z;(\bt_\bm,2t_j)\big) ,\\
 P_{(\bm,j,j);i}\big(z;(\bt_\bm, t_j,t_j)\big) = P_{(\bm,j);i}\big(z;(\bt_\bm, 2t_j)\big) .
\end{gather*}
\end{prop}

We see thus that the repeated application of a 2-step confluent Darboux transformation at the same eigenvalue only serves to modify the deformation parameter. In general, if the two parameters at the repeated $j$ are different, we would have similarly
 \begin{gather*}
 \tau_{(\bm,j,j)}(z;(\bt_\bm,t_j,t_j')= \tau_{(\bm,j)}(z;(\bt_\bm,t_j+t'_j))),\\
 P_{(\bm,j,j);i}(z;(\bt_\bm, t_j,t_j')) = P_{(\bm,j);i}(z;(\bt_\bm, t_j+t'_j).
 \end{gather*}

\begin{proof} We omit again explicit dependence on $z$ and $t_{\bm}, t_j$ if no confusion arises.
 We apply Proposition~\ref{prop:taurec} and Definition~\ref{def:recur} twice to obtain
 \begin{gather*}
 \tau_{(\bm,j,j)}=(1+t_jR_{(\bm,j);jj})\tau_{(\bm,j)} =\left(1+t_jR_{\bm;jj}-\frac{t_j^2 R_{\bm;jj}^2}{1+t_j R_{\bm;jj}}\right)(1+t_jR_{\bm;jj})\tau_\bm\\
\hphantom{\tau_{(\bm,j,j)}}{} =(1+2t_jR_{\bm;jj})\tau_\bm=\tau_{(\bm,j)}\big((z;(\bt_\bm,2t_j)\big), \\
 P_{(\bm,j,j);i}=(1+t_jR_{(\bm,j);jj})P_{(\bm,j);i} -t_jR_{(\bm,j);ij}P_{(\bm,j);j} \\
 \hphantom{P_{(\bm,j,j);i}}{}
 =(1+t_jR_{(\bm,j);jj})\big((1+t_jR_{\bm;jj})P_{\bm;i}-t_jR_{\bm;ij}P_{\bm;i}\big) \\
 \hphantom{P_{(\bm,j,j);i}=}{} -t_j\left(R_{\bm;ij}-\frac{t_j R_{\bm;jj}R_{\bm;ij}}{1+t_jR_{\bm;jj}}\right)P_{\bm;j}
 \\
 \hphantom{P_{(\bm,j,j);i}}{} =(1+2t_jR_{\bm;jj})P_{\bm;i}-2t_jR_{\bm;ij} P_{\bm;j} =P_{(\bm,j);i}(z;(\bt_\bm, 2t_j)).\tag*{\qed}
 \end{gather*}\renewcommand{\qed}{}
\end{proof}

\section{Examples}\label{sec:examples}

To conclude, we present some examples of exceptional Legendre polynomials and orthogonality relations for the cases of $n=1$ and $n=2$.
\subsection{The 1-parameter exceptional Legendre family}

In the 1-parameter case, we have $\bm=({m_1})\in \Nz$, and
\begin{gather*}
 \tau_{m_1}(z;t_{m_1}) = 1+ t_{m_1} R_{{m_1}{m_1}}(z),\\
 P_{{m_1};i}(z;t_{m_1}) = (1 + t_{m_1} R_{{m_1}{m_1}}(z))P_i(z) -
 t_{m_1} R_{i{m_1}}(z)P_{m_1}(z),\qquad i\in \Nz.
\end{gather*}
Note that $P_{{m_1};{m_1}}(z) = P_{m_1}(z)$. The degrees of the other polynomials are
\[ \deg P_{{m_1};i}(z) = i+2{m_1}+1,\qquad i\neq m_1. \]
The corresponding exceptional operator is $T_{m_1}(t_{m_1}) = \hT(\tau_{m_1}(z;t_{m_1}))$
with the latter as per~\eqref{eq:hTdef}.

The polynomial $\tau_{m_1}(z;t_{m_1})$ does not vanish in $[-1,1]$ provided that
\begin{equation*}%\label{eq:const1}
 t_{m_1}>-\frac{1}{R_{m_1m_1}(1)}=-m_1-\frac 1 2.
\end{equation*}
In this case, $\{P_{{m_1};i}(z;t_{m_1})\}_{i\in\Nz}$ is a family of exceptional Legendre polynomials, with orthogonality weight
\begin{gather*}
 W_{m_1}(z;t_{m_1})=\frac{1}{\tau_{m_1}(z;t_{m_1})^2}.
\end{gather*}
The above set is a complete orthogonal polynomial basis of the space $\rL^2([-1,1],W_{m_1}(z;t_{m_1}){\rm d}z)$.
The orthogonality relations are
\begin{gather*}
 \int_{-1}^1 \frac{P_{{m_1};i_1}(u;t_{m_1}) P_{{m_1};i_2}(u;t_{m_1})}{\tau_{m_1}(u;t_{m_1})^2} {\rm d}u
= \frac{2}{1+2i_1}\delta_{i_1i_2} ,\qquad i_1,i_2\in\Nz\backslash\{ m_1\} ,\\
 \int_{-1}^1 \frac{P_{{m_1};{m_1}}(u;t_{m_1})^2}{\tau_{m_1}(u;t_{m_1})^2} {\rm d}u= \frac{2}{1+2{m_1}+2t_{m_1}} .
\end{gather*}
The above example illustrates perfectly the isospectral nature of the
CDT that relates $T(1)\to T_{m_1}(t_{m_1})$. The eigenvalues of the two
operators are the same. As for the eigenfunctions, if $t_{m_1}\neq 0$,
then for $i\neq {m_1}$ they are transformed, but their norms stay the
same. On the other hand, for~$i={m_1}$ the opposite happens: the
eigenfunction does not change but its norm does.

Since $\tau_{m_1}'(z)=t_{m_1} P_{m_1}^2(z)$, it follows that the weight $W_{m_1}(z)$ is a decreasing (resp. increasing) function for $t_{m_1}>0$ (resp. $t_{m_1}<0$), which has $m_1$ saddle points in $[-1,1]$ at the zeros of~$P_{m_1}$, but no local minima or maxima.

For instance, for $m_1=4$, we have
\[ \tau_{4}(z,t_{4}) = 1 + \frac{1}{576} t_{4} \big(64 + 81 z - 540 z^3 + 1998 z^5 - 2700 z^7 + 1225 z^9\big)\]
and the weight is shown in Fig.~\ref{fig:1}.
\begin{figure}[t] \centering
 \includegraphics[width=0.45\textwidth]{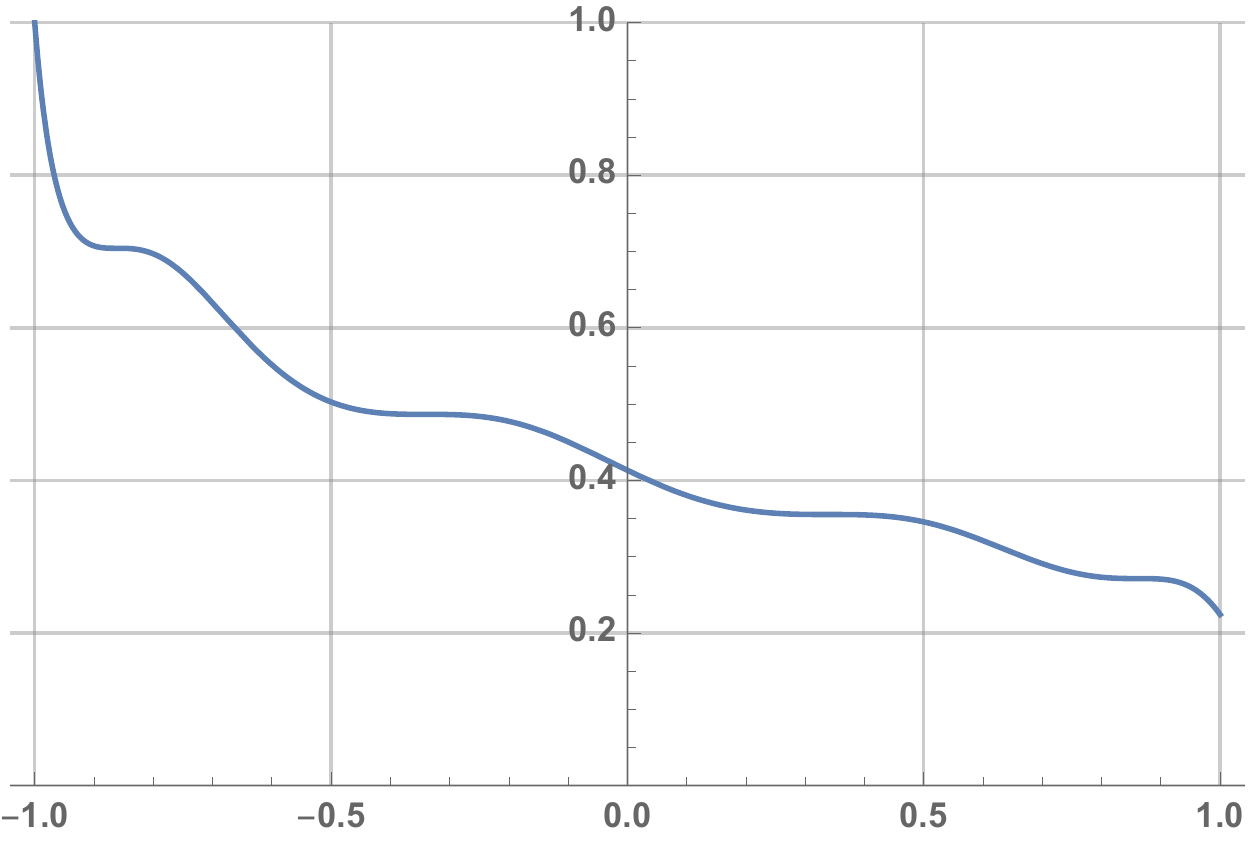}
 \caption{Exceptional Legendre weight $\tau_{m_1}^{-2}(z,t_{m_1})$ for $m_1=4$ and $t_{m_1}=5.2$.} \label{fig:1}
\end{figure}
The first few polynomials for this choice are
\begin{gather*}
P_{4;0}(z,t_{4})=1 + \frac{t_{4}}{144} \big(16 + 135 z^3 - 459 z^5 + 585 z^7 - 245 z^9\big),\\
P_{4;1}(z,t_{4})=z+\frac{t_{4}}{1152}\big(-9 + 128 z + 171 z^2 + 30 z^4 - 1314 z^6 + 2475 z^8 - 1225 z^{10}\big),\\
P_{4;2}(z,t_{4})=P_2(z)-\frac{t_{4}}{576} \big(32 - 96 z^2 + 189 z^3 - 756 z^5 + 1278 z^7 - 1300 z^9 + 525 z^{11}\big),\\
P_{4;3}(z,t_{4})=P_3(z)-\frac{t_{4}}{9216} \big(243 - 1536 z - 3402 z^2 + 2560 z^3 + 3645 z^4 + 7668 z^6 - 17955 z^8 \\
\hphantom{P_{4;3}(z,t_{4})=P_3(z)-\frac{t_{4}}{9216} \big(}{} + 16950 z^{10} - 6125 z^{12}\big),\\
P_{4;4}(z,t_{4})= P_4(z),\\
P_{4;5}(z,t_{4})= P_5(z)+\frac{t_{4}}{9216} \big(243 + 1920 z - 1215 z^2 - 8960 z^3 - 3645 z^4 + 8064 z^5 + 17145 z^6 \\
\hphantom{P_{4;5}(z,t_{4})= P_5(z)+\frac{t_{4}}{9216} \big(}{} - 42255 z^8 + 66171 z^{10} - 50855 z^{12} + 15435 z^{14}\big).
\end{gather*}
We display the above polynomials for $t_{4}=0$ (classical Legendre) and for $t_{4}=5.2$ in Fig.~\ref{fig:3}. Observe that all polynomials undergo a continuous deformation as $t_{4}$ changes, except for \mbox{$P_{4;4}=P_4$} that stays the same.
\begin{figure}[t] \centering
 \begin{tabular}{cc}
 \includegraphics[width=0.43\textwidth]{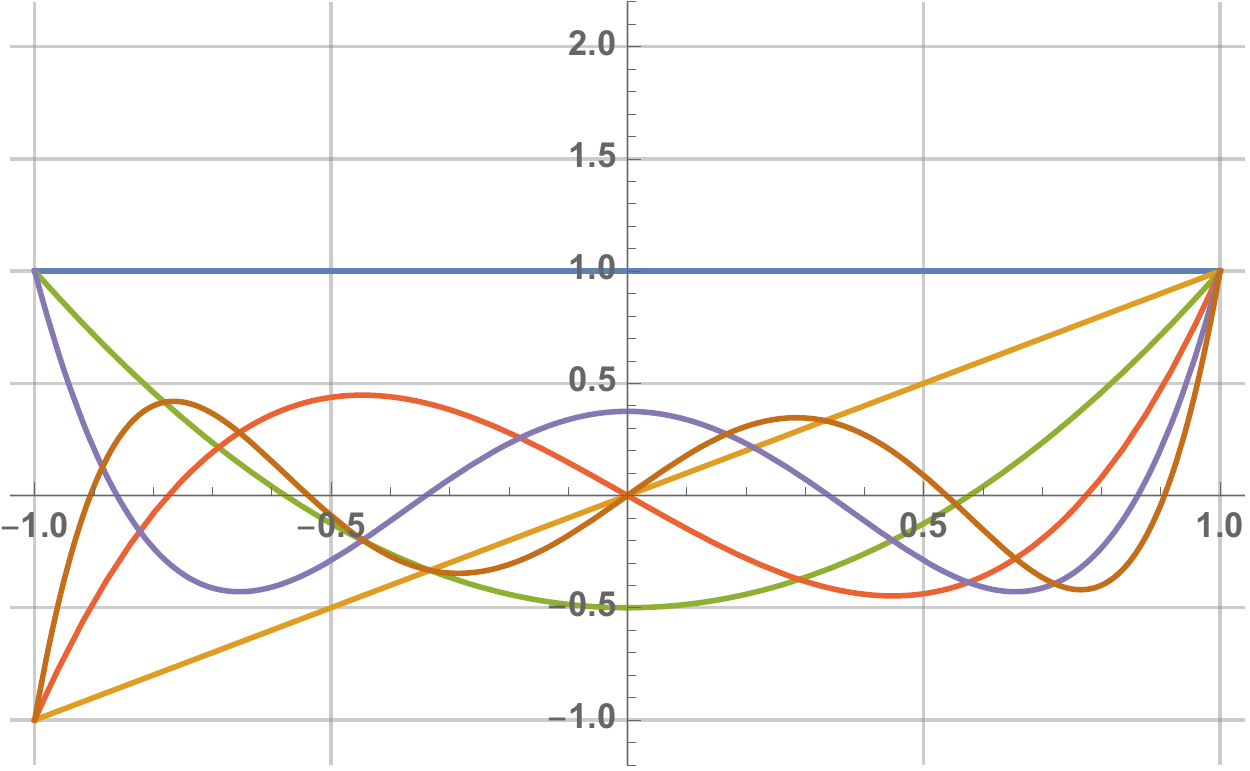} & \includegraphics[width=0.43\textwidth]{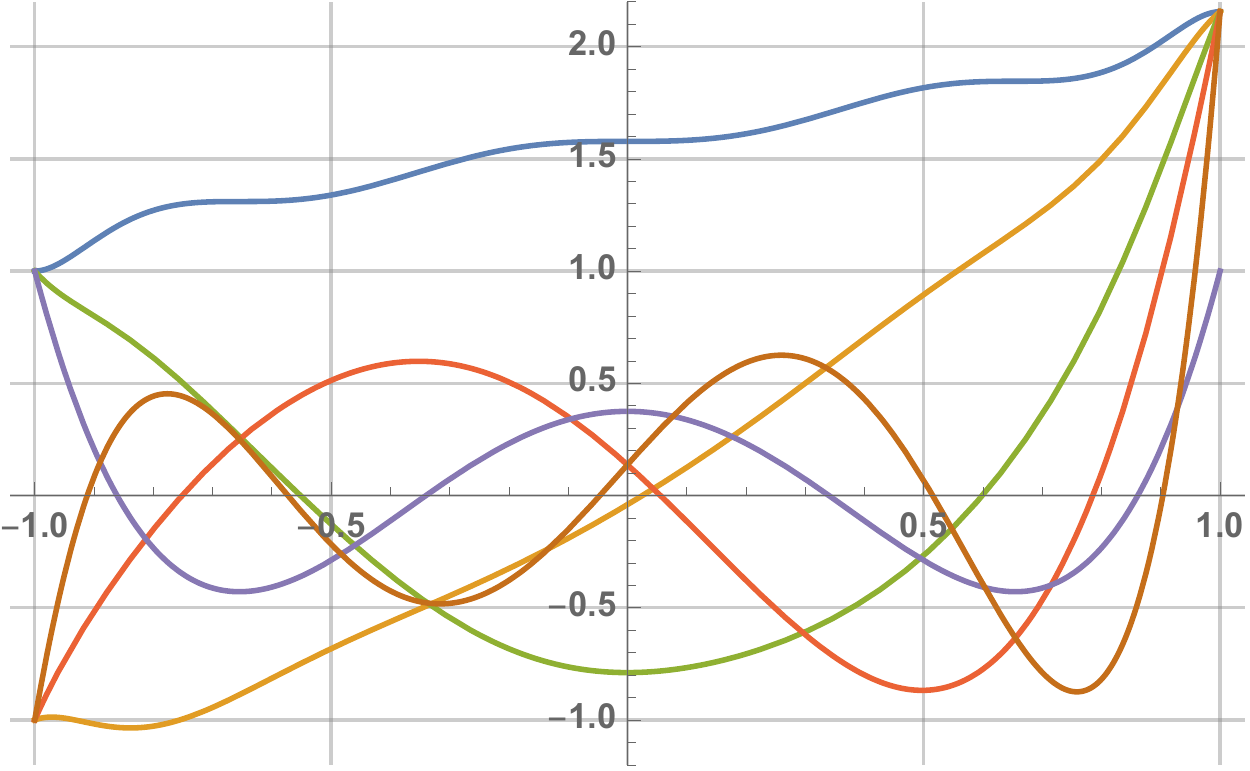}
 \end{tabular}
 \caption{First few exceptional Legendre polynomials $P_{m_1;i}(z;t_{m_1})$ for $m_1=4$, with $t_{m_1}=0$ (left) and $t_{m_1}=5.2$ (right).} \label{fig:3}
\end{figure}

\subsection{The 2-parameter exceptional Legendre family}

In the 2-parameter case, we have $\bm=(m_1,m_2)\in\Nz^2$, $\bt_\bm=(t_{m_1},t_{m_2})$ and
\begin{gather*}
 \tau_\bm(z;\bt_\bm)= \tau_{m_1}(z;t_{m_1})\tau_{m_2}(z;t_{m_2}) -t_{m_1} t_{m_2} R_{m_1m_2}(z)^2,\\
 P_{\bm;i}(z;\bt_\bm)= P_i(z) \tau_\bm(z;\bt_\bm)- P_{m_1}(z) t_{m_1}\tau_{m_2}(z;t_{m_2}) R_{m_1;m_2,i}(z;t_{m_2}) \\
\hphantom{P_{\bm;i}(z;\bt_\bm)=}{} - P_{m_2}(z) t_{m_2}\tau_{m_1}(z;t_{m_1})
 R_{m_1;m_2,i}(z;t_{m_1}), \qquad i\in\Nz,
\end{gather*}
where
\begin{gather*}
 R_{m_1;i_1i_2}(z;t_m)= \int_{-1}^z \frac{P_{m_1;i_1}(u;t_{m_1}) P_{{m_1};i_2}(u;t_{m_1})}{\tau_{m_1}(u;t_{m_1})^2}{\rm d}u\\
\hphantom{R_{m_1;i_1i_2}(z;t_m)}{} =R_{i_1i_2}(z) - \frac{t_{m_1} R_{i_1{m_1}}(z) R_{i_2{m_1}}(z)}{1+t_{m_1} R_{{m_1}{m_1}}} , \qquad i_1,i_2\in\Nz,
\end{gather*}
Supposing that $m_1\neq m_2$ we have
\begin{gather*}
 \deg P_{\bm;i}(z;\bt_\bm)= i+2m_1+2m_2+2,\qquad i\in \Nz\backslash \{ m_1, m_2\} ,\\
 \deg P_{\bm,m_1}(z;\bt_\bm)= m_1+2m_2+1, \\
 \deg P_{\bm;m_2}(z;\bt_\bm) = m_2+2 m_1+1.
\end{gather*}
The polynomial $\tau_{\bm}(z;t_{m_1,t_{m_2}})$ does not vanish in $[-1,1]$ provided that
\begin{gather*}
t_{m_1}>-m_1-\frac 1 2,\qquad t_{m_2}>-m_2-\frac 1 2.
\end{gather*}
In this case, $\{P_{\bm;i}(z;\bt_\bm)\}_{i\in\Nz}$ is a family of exceptional Legendre polynomials, with orthogonality weight
\begin{gather*}
 W_{\bm}(z;\bt_\bm)=\frac{1}{\tau_{\bm}(z;\bt_\bm)^2}.
\end{gather*}
The above set is a complete orthogonal polynomial basis of the space $\rL^2 ([-1,1], W_{\bm}(z;\bt_\bm){\rm d}z)$.

The orthogonality relations are
\begin{gather*}
 \int_{-1}^1 \frac{P_{\bm;i_2}(u;t_\bm)
 P_{\bm;i_2}(u;\bt_\bm)}{\tau_\bm(u;\bt_\bm)^2} {\rm d}u
 = \frac{2}{1+2i_1}\delta_{i_1i_2} ,\qquad i_1,i_2\in\Nz\backslash \{m_1,m_2\} ,\\
 \int_{-1}^1 \frac{P_{\bm;m_1}(u;\bt_\bm)^2}{\tau_\bm(u;\bt_\bm)^2} {\rm d}u
 = \frac{2}{1+2m_1+2t_{m_1}} ,\\
 \int_{-1}^1 \frac{P_{\bm;m_2}(u;\bt_\bm)^2}{\tau_\bm(u;\bt_\bm)^2} {\rm d}u
 = \frac{2}{1+2m_2+2t_{m_2}} .
\end{gather*}
For instance, for $\bm=(m_1,m_1)=(1,2)$, we have
\begin{gather*}\tau_{(1,2)}\big(z;(t_1,t_2)\big)
= 1 + \frac{1}{3} t_{1} \big(1 + z^3\big) + \frac{1}{20} t_{2} \big(4 + 5 z - 10 z^3 + 9 z^5\big)\\
\hphantom{\tau_{(1,2)}\big(z;(t_1,t_2)\big) =}{}
 + \frac{1}{960} t_{1} t_{2} (1 + z)^4 \big(49 - 116 z + 110 z^2 - 36 z^3 + 9 z^4\big).
 \end{gather*}
For certain values of $\bt_\bm=(t_{m_1},t_{m_2})$ the weight is displayed in Fig.~\ref{fig:2}. To the best of our knowledge, this is the first example of an exceptional orthogonal polynomial system whose weight is not monotonic or unimodal.
\begin{figure}[t] \centering
\includegraphics[width=0.45\textwidth]{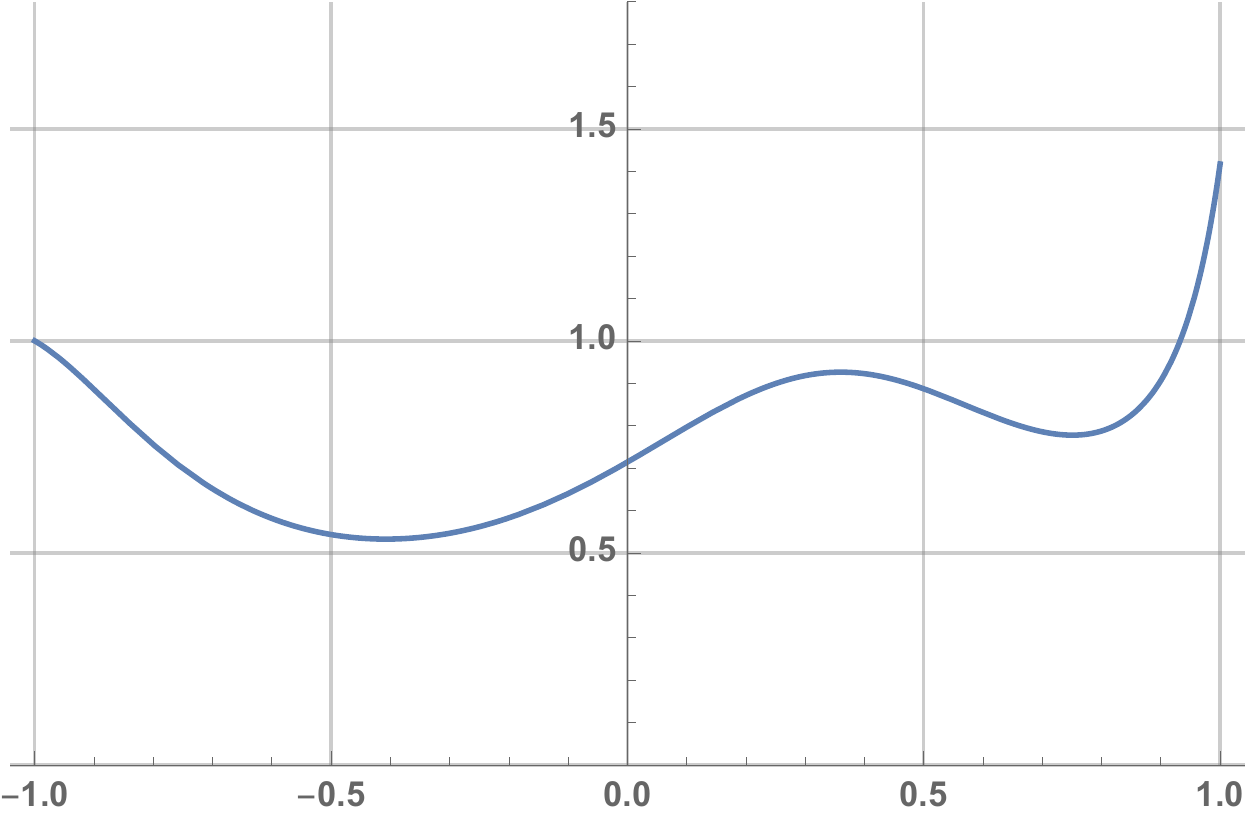}
\caption{Exceptional Legendre weight $\tau_{\bm}^{-2}(z,\bt_\bm)$ for $\bm=(m_1,m_2)=(1,2)$ and $(t_{m_1},t_{m_2})=(2,-1.6)$.} \label{fig:2}
\end{figure}

\section{Summary}
In this paper we show that the class of exceptional orthogonal polynomials is much larger than previously thought. A new construction
based on confluent Darboux transformations leads to new exceptional families with some similarities and differences with respect to the
other exceptional families. In common, they are also Sturm--Liouville problems with rational coefficients, each family is indexed by a set of integers, and it defines a complete basis of polynomial eigenfunctions, whose degree sequence has missing degrees. But as
opposed to the other exceptional families, there are no gaps in the spectrum and the new families contain an arbitrary number of real
deformation parameters, so the construction can be seen as an isospectral deformation of the classical operators. We illustrate the
new construction by describing the full class of exceptional Legendre polynomials, which cannot be derived through the standard
construction. The same method can be applied with minor modifications to the Jacobi operator. A more exhaustive description of these matters will be provided in a~forthcoming publication, together with a discussion of the implications for the classification of exceptional polynomials~\cite{GFGM21}.

\subsection*{Acknowledgements}

MAGF would like to thank the Max-Planck-Institute for Mathematics in the Sciences, Leipzig (Germany), where some of her work took place. DGU acknowledges support from the Spanish MICINN under grants PGC2018-096504-B-C33 and RTI2018-100754-B-I00 and the European Union under the 2014-2020 ERDF Operational Programme and by the Department of Economy, Knowledge, Business and University of the Regional Government of Andalusia (project FEDER-UCA18-108393).

\pdfbookmark[1]{References}{ref}
\LastPageEnding

\end{document}